%% file: blrs2.0.tex
\documentclass[12pt]{amsart}
\usepackage{amssymb}
\allowdisplaybreaks

\usepackage{color}
\usepackage{graphics}
\usepackage{amsmath,amscd}
\usepackage{amsthm}
\usepackage[all,cmtip]{xy}
\usepackage{tikz}
\usepackage{graphicx}
\usepackage{epstopdf}
\usepackage{mathtools}
\usepackage{amssymb}
\usepackage{amsbsy}
\usepackage{bm}
\usepackage{tikz-cd}
\usepackage{enumerate}
\usepackage{enumitem}
\usepackage{float}
\usepackage[inkscapearea=page]{svg}
\theoremstyle{definition}
\newtheorem{exmp}{Example}[section]
\usepackage[colorlinks=green,linkcolor=blue,citecolor=red,urlcolor=red]{hyperref}

\newcommand{\Z}{\mathbb{Z}}

\newcommand{\D}{\mathbb{D}}

\newcommand{\s}{\mathbb{S}}

\newcommand{\id}{\textrm{id}}

\newcommand{\Diff}{\operatorname{Diff}}

\newcommand{\ob}{\textrm{OB}}

\newtheorem{proposition}{Proposition}[section]
\newtheorem{theorem}[proposition]{Theorem}
\newtheorem{definition}[proposition]{Definition}
\newtheorem{lemma}[proposition]{Lemma}

\newtheorem{corollary}[proposition]{Corollary}
\newtheorem{remark}[proposition]{Remark}

\begin{document}
	
	\title{A note on codimension $2$ spun embedding}
	
	\subjclass{}
	
	\keywords{embedding, open book, 3-manifold}

	\author{Sneha Banerjee}
	\address{IAI TCG CREST Kolkata, and Academy of Scientific Research and Innovation, Gazhiabad}
	\email{sneha.banerjee.121@tcgcrest.org}
	
	\author{Shital Lawande}
	\address{IAI TCG CREST Kolkata, and Ramakrishna Mission Vivekananda Education and Research Institute, Belur Math}
	\email{shital.lawande@tcgcrest.org}
	
	\author{Subhadeep Rana}
	\address{IAI TCG CREST Kolkata, and Academy of Scientific Research and Innovation, Gazhiabad}
	\email{subhadeep.rana.125@tcgcrest.org}
	
	\author{Kuldeep Saha}
	\address{IAI TCG CREST Kolkata, and Academy of Scientific Research and Innovation, Gazhiabad}
	\email{kuldeep.saha@gmail.com, kuldeep.saha@tcgcrest.org}

	\begin{abstract}
		
		We prove that if a closed manifold $B$ is a connected component of the binding of an open book decomposition of a manifold $M$, then every open book decomposition of $B$ spun embeds in $M$. As an application, we prove that every open book decomposition of a simply connected spin $5$-manifold spun embeds in $\s^7$ and every $3$-dimensional open book spun embeds in $\s^5$. We also define a notion of spun embedding for Morse open books.
		
	\end{abstract}
	
	\maketitle

	\section{Introduction}	
	
	An open book decomposition of a manifold $M^m$ is a pair $(V^{m-1},h)$, such that $M^m$ is diffeomorphic to the quotient space $\mathcal{MT}(V^{m-1}, h) \cup_{id} \partial V^{m-1} \times D^2$. Here, $V^{m-1}$ is called the \emph{page} and $\partial V$ is called the \emph{binding} of the open book. The map $h$, called \emph{monodromy}, is a diffeomorphism of $V^{m-1}$ that restricts to identity near the boundary $\partial V$, and $\mathcal{MT}(V^{m-1}, h)$ denotes the mapping torus of $h$. We denote the total space of such an open book by $\textrm{OB}(V,h)$. 
	
	\noindent An embedding $f$ of $\textrm{OB}(V_1,h_1)$ in $\textrm{OB}(V_2,h_2)$ is called a \emph{spun embedding} (or an \emph{open book embedding}), if $f$ takes the page $V_1$ to the page $V_2$ and $f$ is compatible with the monodromies : $h_2 \circ f = f \circ h_1$, up to isotopy. We say $M^n$ spun embeds in $N^{n+k}$ if there exists an open book decomposition of $M$ that spun embeds in an open book decomposition of $N$. In recent times, the study of codimension $2$ \emph{spun embedding} (or \emph{open book embedding}) has attracted much attention (eg. \cite{EF}, \cite{EL}, \cite{pps}, \cite{GPS}, \cite{Saha} ,\cite{ls1}). The notion of spun embedding of an open book decomposition is strictly stronger than ordinary smooth embedding. Obstructions to spun embedding an open book in a given open book has been discussed in \cite{ls1}. Etnyre and Lekili \cite{EL} gave the first constructions of codimension $2$ spun embeddings of $3$-manifolds by embedding pages as fibers in a Lefschetz fibration. Later, an alternate construction of spun embedding was given by Pancholi, Pandit and the fourth author \cite{pps}, based on the observation that there is a proper embedding of a annulus $\mathcal{A}$ in $\D^4$ such that, the Dehn twist along the core of $\mathcal{A}$ is induced by an isotopy of $\D^4$, relative to the boundary. This construction, called the \emph{annulus trick}, was later adapted to the setting of non-orientable $3$-manifolds by Ghanawat, Pandit and Selvakumar \cite{GPS}. In this note, we observe a generalization of this construction and give some applications.

	\begin{theorem} \label{thm0}
		
		Let $M$ be a closed orientable $n$-manifold with an open book decomposition $\ob(P^{n-1},\rho)$. Let $B^{n-2}$ be a connceted component of the binding $\partial P$ and let $\ob(V_B,\phi_B)$ be an open book decomposition of $B$. Then, $\ob(V_B, \phi_B^k)$ spun embeds in $\ob(P^{n-1},\rho)$ for all $k \in \Z$.
		
	\end{theorem}

	\noindent Here, $\phi_B^k$ denotes the $k$-times composition of $\phi_B$. An immediate corollary of Theorem \ref{thm0} is the following.

	\begin{corollary} \label{cor1}
		Every open book decomposition of $\s^{n-2}$ spun embeds in $\s^n = \ob(\D^{n-1}, \id)$ for $n \geq 4$.
	\end{corollary}
	
	\noindent More generally, one can think of any codimension $2$ fibered submanifold of $\s^m$. In particular, every $(2n-1)$-dimensional Brieskorn manifold appears as the binding of an open book decomposition of $\s^{2n+1}$. 
	
	 \noindent Let $Q^5$ be a simply connected spin manifold. Kwon and Van Koert (Corollary $3.9$ in \cite{KoVan}) have shown that every simply connected spin $5$-manifold is a connected sum of Brieskorn manifolds. Thus, $M = B_1 \# B_2 \# \cdots \# B_l$ for some $l \in \Z_{>0}$, where the $B_i$s are $5$-dimensional Brieskorn manifolds. Let $\ob(V_i, \phi_i)$ be an open book decomposition of $\s^7$ with binding $\partial V_i = B_i$. Then, $\ob(V_1 \#_b V_2 \#_b \cdots \#_b V_l, \phi_1 \circ \phi_2 \circ \phi_l) = \ob(V_1,\phi_1) \# \ob(V_2,\phi_2) \# \cdots \# \ob(V_l,\phi_l)$ (see section $2.4$ in \cite{DGK}) is an open book decomposition of $\s^7$ with binding $Q^5$.  Thus, Theorem \ref{thm0} implies the following.
	
	\begin{corollary} \label{cor2}
		
		Let $Q$ be a closed simply connected spin $5$-manifold. Then, every open book decomposition of $Q$ spun embeds in $\s^7$ with a fixed page bounded by $Q$.
		
	\end{corollary}
	
  \noindent Saeki (Theorem $6.1$ in \cite{Saeki}) has shown that every closed oriented $3$-manifold can be realized as the connected binding of an open book decomposition of $\s^5$. Together with Theorem \ref{thm0} this implies the following.

	\begin{corollary}\label{cor3}
		
		Let $M^3$ be a closed orientable $3$-manifold. Then, every open book decomposition of $M^3$ spun embeds in $\s^5$ with a fixed page bounded by $M$.
		
	\end{corollary}
	
	\noindent We note that an analogue of \ref{cor2} and \ref{cor3} for contact manifolds can not hold. It is known by the work of Thurston--Winkelenkemper and Giroux that every contact structure on a $3$-manifold corresponds to an open book decomposition. By Kasuya \cite{Kasuya}, a contact manifold $(M,\zeta)$ admits a contact embedding in a contact $\s^5$ only if its first Chern class $c_1(\zeta)$ vanishes. However, in general, it is not true that every spun embedding of a $3$-manifold in $\s^5$ corresponds to a contact embedding of the corresponding contact $3$-manifold. In particular, many of the spun embeddings given by Corollary \ref{cor2} and Corollary \ref{cor3} are examples of codimension $2$ proper embeddings (embedding between pages of open books) which are not isotopic to a symplectic embedding.
	
	\vspace{0.25cm}

	\noindent It is not known yet if there exists a fixed open book decomposition of $\s^5$ where all $3$-dimensional open books admit spun embedding. We call such an open book decomposition of $\s^5$ \emph{universal}. A recent work (Theorem $1.4$ in \cite{ls1}) by the second and the fourth author shows that none of Saeki's open books (given by Theorem $6.1$ in \cite{Saeki}) can be universal.

	\vspace{0.25cm}

	Theorem \ref{thm0} is proved using a simple generalization of the \emph{Hopf annulus trick} (see section \ref{genhopftrick}), which we call the \emph{page trick}. In section \ref{pagetrickap}, we discuss some more applications of the page trick to construct spun embeddings. In particular, the recent finding of various new open book decompositions by Hsueh \cite{hsueh2}, and a generalization of the notion of plumbing by Ozbagci and Popescu-Pampu \cite{OP} give us a large class of novel examples where the page trick can be applied.  
	
	Finally, in section \ref{morseob}, we discuss a notion of spun embedding for \emph{Morse open books} and describe a version of the page trick in this context (see Proposition \ref{morsepagetrick} and Example \ref{brancexmp}).

	Unless stated otherwise, we always work in the category of smooth manifolds.

	\section{preliminaries} \label{prelim}
	\subsection{Open books and embeddings} An open book decomposition of a closed $(2n+1)$-manifold $M$ consists of a codimension $2$ closed submanifold $B$ and a fibration map $\pi : M \setminus B \rightarrow \s^1$, such that in a tubular neighborhood of $B \subset M$, the restriction map $\pi : B \times (\D^2 \setminus \{0\}) \rightarrow \s^1$ is given by $(b,r,\theta) \mapsto \theta$. The fibration $\pi$ determines a unique fiber manifold $N^{2n}$ whose boundary is $B$. The closure $\bar{N}$ is called the \emph{page} and $B$ is called the \emph{binding}. The monodromy of the fibration map $\pi$ determines a diffeomorphism $\phi$ of $\bar{N}$ such that $\phi$ is identity near the boundary $\partial \bar{N}$. In particular, $M = \mathcal{MT}(\bar{N}, \phi) \cup _{id, \partial} \partial \bar{N} \times \D^2$. We denote such an open book decomposition of $M$ by $\textrm{OB}(\bar{N},\phi)$. The map $\phi$ is called the \emph{monodromy} of the open book. This description of an open book decomposition in terms of its page and monodromy is known as an \emph{abstract open book}. However, these two notions are equivalent and we will always describe an open book decomposition of a manifold as an abstarct open book. Two open book decompositions with the same page are \emph{equivalent} if their monodromies are isotopic, relative to a collar neighborhood of the binding in a page.

	\begin{definition}[Spun embedding]
		
		We say, $\textrm{OB}(\Sigma_1,\phi_1)$ \emph{spun embeds} in $\textrm{OB}(\Sigma_2,\phi_2)$, if there exists a proper embedding $g : (\Sigma_1,\partial \Sigma_1) \rightarrow (\Sigma_2,\partial \Sigma_2)$ such that $g \circ \phi_1$ is isotopic to $\phi_2 \circ g$, relative to the boundary. A spun embedding is also known as an \emph{open book embedding} in the literature.
		
	\end{definition}
	
	\noindent By embedding $\partial \Sigma_1 \times \D^2$ in  $\partial \Sigma_2 \times \D^2$ via the map $(b,(r,\theta)) \mapsto (g(b), (r, \theta))$, we get a genuine embedding between the total spaces of the open books.  
	
	\subsection{Seifert type embedding} Let $\Sigma^{n-1}$ be an oriented compact submanifold of an oriented closed manifold $M^n$ with non empty boundary $\partial \Sigma$. Consider the product manifold $M \times [a,b]$ ($a < b)$. We construct a proper embedding $f$ of $\Sigma$ in $M \times [a,b]$ in the following way. We isotope the interior of $\Sigma \times \{a\} \subset M \times \{a\}$ into the interior of $M \times [a,b]$, fixing the boundary $\partial \Sigma \times \{a\}$ pointwise. The resulting embedding $f$ of $\Sigma$ is such that $f(\Sigma) \cap (M \times \{b\}) = \Sigma \times \{b\}$ and $f(\Sigma) \cap (M \times \{(1-t) a + t b\}) = \partial \Sigma \times \{(1-t) a + t b\}$ for all $t \in [0,1)$. See Figure \ref{stfig}. We call $f$ a \emph{Seifert type embedding} of $\Sigma$ in $M \times [a,b]$.

	\begin{figure}[htbp] 
		
		\centering
		\def\svgwidth{10cm}
		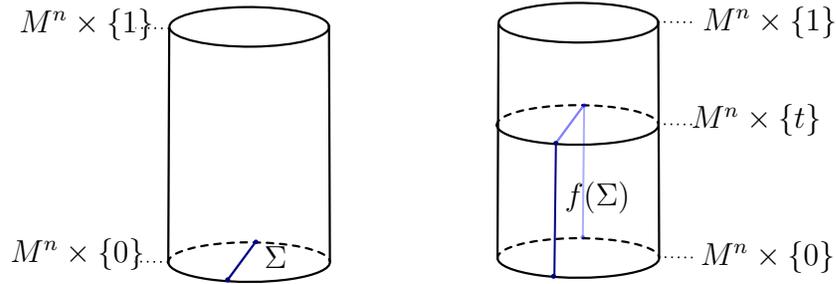
		\caption{Isotoping $\Sigma \subset M \times \{0\}$ to a Seifert type embedding in $M \times [0,1]$.}
		\label{stfig}
		
	\end{figure}
	
	\subsection{Isotopy via the flow of an open book} Given an open book decomposition $\ob(\Sigma, \phi)$ of a manifold $M$, there is a flow on $M$ whose time-$1$ map takes a page to itself with diffeomorphism $\phi$. The vector field corresponding to this flow is transverse to the interior of each page. In a neighborhood of the binding, $(\partial \Sigma \times \D^2, (x,r,\theta))$, it is given by $r \cdot \frac{\partial}{\partial \theta}$. In particular, the flow fixes the binding pointwise. Thus, there is an isotopy $\psi_t$ ($t \in [0,1]$) of $M$ such that $\psi_0 = \id_M$ and $\psi_1|\Sigma = \iota \circ \phi$, where $\iota$ is an embedding of $\Sigma$ in $M$ as a page of the open book.  
	
	\subsection{The annulus trick}	A Hopf annulus is an embedding of an annulus in $\mathbb{S}^3$ with boundary a Hopf link. A Hopf annulus is called positive (or negative) if the linking number of its boundary is $1$ (or $-1$). Let $\tau_c$ denote the Dehn twist along the core curve of an annulus. 
	
	\begin{figure}[!htb]
		\centering
		\def\svgwidth{6cm}
		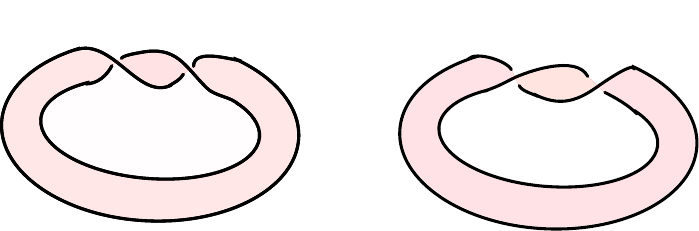
		\caption{Hopf annuli}
		\label{fig-01}
	\end{figure}

	\begin{proposition}[The Hopf annulus trick \cite{hy},\cite{pps}]\label{pr1}
		There exists a Seifert type embedding $f$ of an annulus in $\s^3 \times [-1,1]$ such that there exists an isotopy $\Gamma_s$ ($s \in [0,1]$) of $\s^3 \times [0,1]$, relative to boundary, such that $\Gamma_0 = \id$ and $\Gamma_1 \circ f = f \circ \tau_c$.
	\end{proposition}

	The main idea to prove Proposition \ref{pr1} is to observe that $\s^3$ admits an open book decomposition with page a Hopf annulus $\mathcal{A} \subset \s^3$ and monodromy a Dehn twist along the core curve $c$ of $\mathcal{A}$. Consider a Seifert type embedding $f$ of an annulus in $\s^3\times [0,1]$ corresponding to $\mathcal{A} \subset \s^3$. Let $\Psi_t$ be the isotopy of $\s^3$ such that $\Psi_1$ realizes the Dehn twist monodromy on $\mathcal{A}$ via the flow of the open book $\textrm{OB}(\mathcal{A},\tau_c)$. One can then define an isotopy $\Gamma_s$ ($s\in[0,1]$) of $\s^3 \times  [-1,1]$ such $\Gamma_0 = \id_{\s^3 \times [-1,1]}$ and $\Gamma_1|_{\s^3 \times \{0\}} = \Psi_1.$

	$$
	\Gamma_s (x,t) 
	=
	\left\{
	\begin{array}{ll}
		\Psi_{s(1-t)}(x)  & \mbox{if } t \geq 0 \\
		\Psi_{s(t+1)}(x) & \mbox{if } t \leq 0
	\end{array}
	\right . 
	$$
	
	\noindent It is now easy to see that $\Gamma_1 \circ f = f \circ \tau_c$. By applying $\Gamma_1$ repeatedly on $\s^3 \times [-1,1]$, one can induce the diffeomorphism $\tau_c^k$ for any $k \in \Z$.

	\subsection{A generalization of the annulus trick} \label{genhopftrick}
	
	The Hopf annulus trick has a straightforward generalization for any open book decomposition. Let $N^n = \ob(V^{n-1},\phi)$. Consider the Seifert type embedding $h$ of $V$ in $N\times [-1,1]$ corresponding to the embedding of $V$ in $N$ as a page of the open book. We can follow the proof of the Hopf annulus trick (described in the previous section) verbatim by replacing $\s^3$, $\mathcal{A}$ and $\tau_c$ with $N$, $V$ and $\phi$, respectively. Thus, we get an isotopy $H_s$ ($s \in [0,1]$) of $N \times [-1,1]$ such that $H_0 = \id_{N \times [0,1]}$ and $H_1 \circ h = h \circ \phi$. Similar to the previous case, for $k \in \Z$, the diffeomorphism $\phi^k$ can be induced by a repeated application of $H_1$. 
	
	\begin{proposition}[The page trick] \label{pagetrick}
		
		There exists a Seifert type embedding $h$ of $V$ in $N \times [-1,1]$ such that there is an isotopy $\tilde{\Gamma}_s$ ($s \in [0,1]$) of $N \times [-1,1]$, relative to boundary, such that $\tilde{\Gamma}_0 = \id$ and $\tilde{\Gamma}_1 \circ h = h \circ \phi$. 
	
	\end{proposition}
	
	\noindent If $N$ is a boundary component of an oriented manifold $W^{n+1}$, then we can extend $\Gamma_1$ to a diffeomorphism of $W$ via the identity map, and $\Gamma_1$ is also isotopic to identity, where the isotopy is supported in a collar neigborhood of $N$.

	\section{Proof of the main theorem}
	
	\begin{proof}[Proof of Theorem \ref{thm0}]  Recall that $M^n = \ob(P^{n-1}, \rho)$, where $\rho$ restricts to the identity map on a collar neighborhood $\mathcal{C}(B)$ of $B = \partial P$ in $P$. Let us identify this collar neighborhood with $B \times [-1,1]$ such that $B \times \{-1\} = \partial P$. Consider an open book decomposition $\ob(V_B,\phi_B)$ of $B$ and let $f$ be the Seifert type embedding of $V_B$ in $B \times [-1,1]$ such that $f(\partial V_B) \subset B \times \{-1\}$.  This gives a proper embedding of $V_B$ in $P$. Given an integer $k$, by Proposition \ref{pagetrick}, there exists a diffeomorphism $\Gamma_k$ of $P$ with the following properties.
		
		\begin{enumerate}
			\item  $\Gamma_k$ is supported on $\mathcal{C}(B)$ and $\Gamma_k$ is isotopic to the identity map, relative to $P \setminus \mathcal{C}(B)$.
			
			\item $\Gamma_k \circ f = f \circ \phi_B^k$.   
		\end{enumerate}   
		
		\noindent Since $\rho|_{\mathcal{C}(B)} = \id_P$, $\ob(V_B, \phi^k_B)$ spun embeds in $\ob(P,\rho)$.
		
	\end{proof}

	\section{The page trick in a more general setup} \label{pagetrickap} 
	
	Let $M^n$ be a component of the binding of on open book $\ob(\Sigma,\phi)$. Let $\ob(P_1^{n-1}, \rho_1)$ and $\ob(P_2^{n-1}, \rho_2)$ be two open book decompositions of a closed oriented manifold $M^n$. Let $g_i$ denote the embedding of $P_i$ in $M$ as a fixed page of the open books for $i = 1,2$. Let $V$ be a compact bounded submanifold of $M$ such that $g_1(P_1) \cup g_2(P_2)$ is a submanifold of $V$. 
	
	\begin{lemma} \label{genpagetrick}
	$\ob(V, \eta)$ spun embeds in $\ob(\Sigma, \phi)$ for any $\eta \in \langle \rho_1, \rho_2\rangle \subset \Diff_\partial(V)$.
	\end{lemma}

	\begin{exmp}\label{expg1}
		
	The $5$-sphere admits an open book decomposition with page the disk cotangent bundle of $\s^2$ and monodromy a Dehn-seidel twist : $\ob(DT^*\s^2, \tau_2)$ (see \cite{Saha}). Another open book on $\s^5$ has page is given by $\ob(\s^1 \times \D^3 \#_b \s^3 \times [0,1], \sigma)$, where $\sigma$ is the \emph{push} map (see \cite{hsueh1}, \cite{hsueh2}). Let $g_1$ and $g_2$ be the embeddings of $P_1 = DT^*\s^2$ and $P_2 = \s^1 \times \D^3 \#_b \s^3 \times [0,1]$, respectively, in $\s^5$. Then, $g_1(P_1) \#_b g_2(P_2)$ gives a hypersurface in $\s^5 = \s^5 \# \s^5$. By Lemma \ref{genpagetrick}, for any diffeomorphism $\eta$ of $g_1(P_1) \#_b g_2(P_2)$ generated by $\tau_2$ and $\sigma$, $\ob(g_1(P_1) \#_b g_2(P_2), \eta)$ spun embeds in $\ob(\D^6,\id)$.  \noindent Hsueh \cite{hsueh2} has given a large family of new open book decompositions on $\s^n$. Using his open books, one can generate many more examples as the one above.  
		
	\end{exmp} 
	
	\begin{remark}
		It follows from Theorem \ref{thm0} that $M_k = \ob(\s^1 \times \D^2 \#_b \s^2 \times [0,1], \sigma^k)$ spun embeds in $\ob(\D^5,\id) = S^6$. Since $H_1(M_k) = \Z_k$, by a result of Cochran (see Theorem $6.2$ in \cite{coch}), $M_k$ also embeds in $\s^5$ for all $k \geq 0$. It would be interesting to know whether $M_k$ spun embeds in $\s^5$ for all $k \geq 2$. 
	\end{remark}
	
	\begin{exmp}\label{expg2}
		
	In example \ref{expg1}, we obtained the page as boundary connected sum of pages of two open books. Suppose we have two pages $P^n_1$ and $P^n_2$ such that there exist relative embeddings $e_1 : (\D^k \times \D^{n-k}, \partial \D^k \times \D^{n-k})  \rightarrow (P_1, \partial P_1)$ and $e_2 : (\D^{n-k}\times \D^k, \partial \D^{n-k} \times \D^k) \rightarrow (P_2, \partial P_2)$. Then one can define the plumbing of $P_1$ and $P_2$ by gluing the images of $e_1$ and $e_2$ via the map $e_1(\vec{x},\vec{y}) \mapsto e_2((-1)^k \vec{y}, \vec{x})$. Ozbagci and Popescu-Pampu has defined a generalized notion of plumbing where two manifolds such as $P_1$ and $P_2$ can be \emph{summed} using a common \emph{patch} (see section $7$ in \cite{OP}).  
		
	\end{exmp}

	\begin{figure}[htbp] 
		
		\centering
		\def\svgwidth{10cm}
		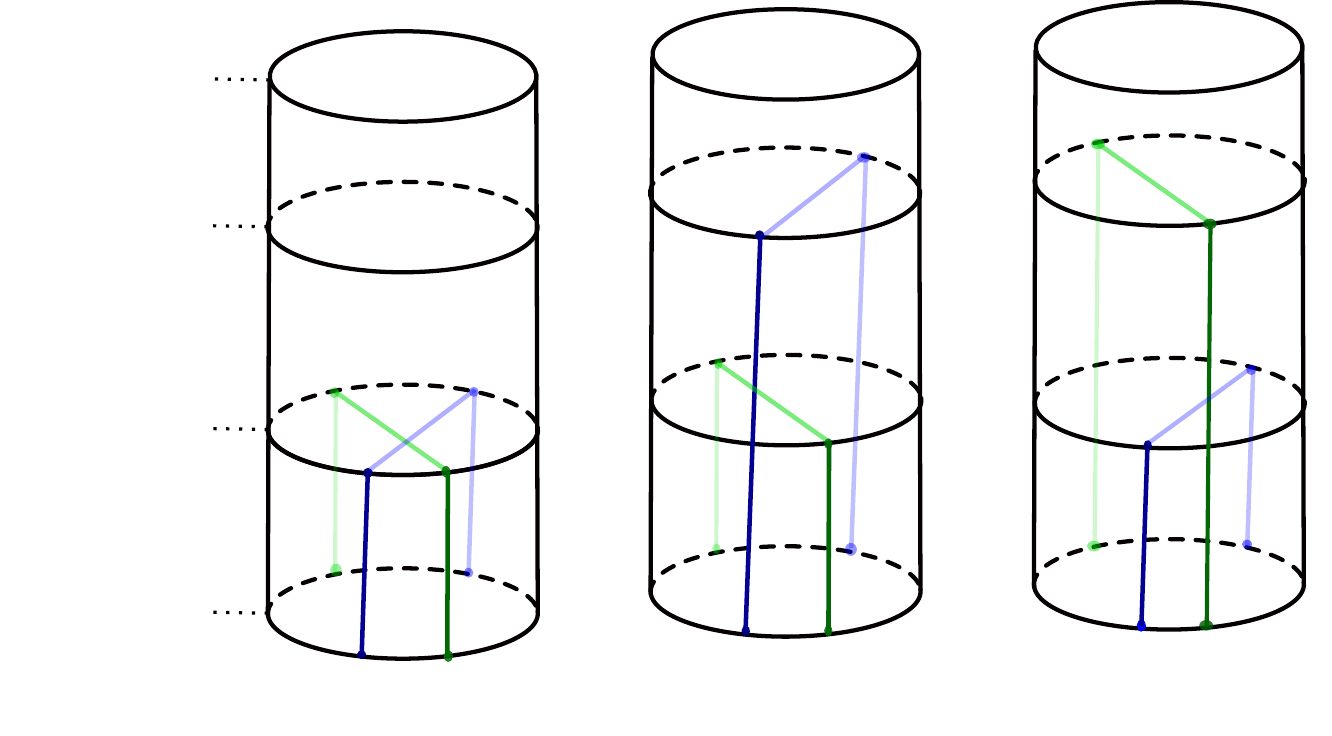
		\caption{The purple and green lines represent $g_1(P_1)$ and $g_2(P_2)$, respectively,  in $M \times [-2,1]$ and $s$ is a coordinate on the second factor of $\Sigma \times [0,1]$. We first isotope $g_1(P_1)$ to induce $\rho_1$ at $s = \frac{1}{3}$, and then isotope $g_2(P_2)$ to induce $\rho_2$ at $s =\frac{2}{3}$. This induces the monodromy $\rho_1 \circ \rho_2$ on $V$.}
		\label{gptfig}
		
	\end{figure}

	\begin{proof}[Proof of Lemma \ref{genpagetrick}]  We identify a collar neighborhood of $M \subset \Sigma$ with $M \times [-2,2]$ such that $\partial \Sigma = M \times \{2\}$. We take a Seifert type proper embedding of $V$ in $M \times [-2,-1]$. Suppose, we want to induce the map $\rho_1$ on $V$ first. For that we simply isotope the submanifold $g_1(P_1) \times \{-1\}$, relative to its boundary, such that after the isotopy $g_1(P_1) \times \{-1\}$ becomes a Seifert type embedding $f_1$ of $P_1$ in $M \times [-1,0]$. We then apply the page trick in the neighborhood $M \times [-1,1]$ to induce the map $\rho_1$, and isotope $f_1(P_1)$ back to $g_1(P_1) \times \{-1\}$. Similarly, one can induce the map $\rho_2$ on the properly embedded $V$ in $M \times [-2,2]$. See Figure \ref{gptfig}. Let $(\Sigma \times [0,1], (x,\theta))$ be a coordinate on the mapping cylinder of $\ob(\Sigma, \phi)$. Say, $\eta$ has a presentation as a word generated by $\rho_1$ and $\rho_2$ with $l$ many letters. We divide the mapping cylinder in $l$ equal part $\Sigma \times [\frac{i}{l}, \frac{i+1}{l}]$ ($0 \leq i \leq l-1$), and apply an isotopy of $\Sigma$ on each part to induce each of the letters. Thus, $\ob(V,\eta)$ spun embeds in $\ob(\Sigma, \phi)$. 
		
	\end{proof}

	\noindent Although Lemma \ref{genpagetrick} is stated for two open books, its proof shows that a similar statement holds true for any number of open books. The page trick can be used for constructing $5$-dimensional open books with connected binding where all $3$-dimensional open books with connected binding admit spun embedding (see Theorem $4.4$ in \cite{ls1}). Examples of such $5$-dimensional open books were first constructed by Etnyre and Lekili \cite{EL} using mapping torus of a Lefschetz fibration.

	\section{The page trick for Morse open books} \label{morseob}
	
		A \emph{Morse open book} is similar to an open book $(M^n, B^{n-1},\pi)$, except that the fibration map $\pi : M \setminus B \rightarrow \s^1$ may have singular fibers, with singularities in the interiror of a fiber.
		
	   \noindent  Let $K^{n-2}$ be a closed oriented submanifold of a closed oriented manifold $L^n$ with trivial normal bundle and let $K$ be nullhomologous in $L$. Then, there exists a hypersurface $V^{n-1} \subset L$ such that $\partial V = K$. Given such a \emph{Seifert surface} $V$, there exists a Morse open book on $L$ with a regular page $V$ (see remark $9.6$ in \cite{OP}). In particular, $L \setminus V$ can be seen as a relative cobordism $W^n$ from $V$ to itself. Let $\psi$ be the monodromy map of the regular page $V$. Let $h : W^n \rightarrow [0,1]$ be a Morse function with $h^{-1}(0) = h^{-1}(1) = V$ and $h$ has critical points only in the interior of $W$. According to Milnor's terminology, $h$ is a Morse function on the \emph{triad} $(W, h^{-1}(0), h^{-1}(1))$. The monodromy $\psi$ is then a diffeomorphism (relative to boundary) between $h^{-1}(0)$ and $h^{-1}(1)$. In this case, the analogue of the mapping torus of an open book is the quotient space $\mathcal{M}(W,h,\psi) = \frac{W}{x \sim \psi(x)}$. Thus, the triple $(W,h,\psi)$ characterizes a Morse open book structure on $L$. 
	   
	   \begin{definition}[spun embedding for Morse open books]
	   	
	   	Let $L_1^n$ and $L_2^{n+k}$ be two closed oriented manifolds. Let $(W_1,h_1, \psi_1)$ and $(W_2,h_2,\psi_2)$ be Morse open book decompositions of $L_1$ and $L_2$, respectively. We say $(W_1, h_1, \psi_1)$ spun embeds in $(W_2,h_2,\psi_2)$ if there exists an embedding $\iota$ of $L_1$ in $L_2$ such that the followings hold.

	   	\begin{enumerate}
	   		\item $h_2 \circ \iota$ restricts to a Morse function on the triad $(W_1, h_1^{-1}(0), h_1^{-1}(1))$.
	   		
	   		\item  The embedding $\iota$ takes a regular page of $h_1$ to a regular page of $h_2$ and $\iota \circ \psi_1 | _{h_1^{-1}(1)}= \psi_2 \circ \iota|_{h_1^{-1}(1)}$. 
	   		
	   	\end{enumerate}
	   	
	   \end{definition}
	   
	   \noindent Let $(W, h, \phi)$ be a Morse open book decomposition of a closed oriented manifold $N$. We take a Seifert type embedding of a regular page $V$ (i.e. a regualr level set of $h$) in $N \times [-1,0] \subset N \times [-1,1]$ such that $\partial \bar{V} \subset N \times \{-1\}$. Let $f$ be the inclusion map of $\mathcal{M}(W,h,\psi)$ in $N \setminus V$. Let $f_t = f|_{ h^{-1}(t)}$ for $t \in [0,1]$. Note that restricted to the regular page $V$, $f_1 = f_0 \circ \phi$. We now define a proper embedding $F$ of $W$ in $N \times [-1,1] \times [0,1]$. We first embed $W$ in $N \times \{0\} \times [0,1]$ by an embedding $\hat{f}$ such that $\hat{f}(x) = (f_t(x),0,t)$ for $x \in h^{-1}(t)$. Since $f_t|_{\partial \bar{V}} = \id_{\partial \bar{V}}$ for all $t \in [0,1]$, we can extend this to a proper embedding by adding the cylindrical region $\partial \bar{V} \times [-1,0] \times \{t\} \subset N \times [-1,1] \times [0,1]$.  We summerize the discussion above in the following proposition.

		\begin{proposition} \label{morsepagetrick}
			
		Let $(W, h, \phi)$ be a Morse open book decomposition of a closed oriented manifold $N$. Then, there exists an embedding $F : W \rightarrow N \times [-1,1] \times [0,1]$ with the following property.
		
		\begin{enumerate}
			
			\item The projection map $N \times [-1,1] \times [0,1] \rightarrow [0,1]$ restricts to a Morse function on the triad $(F(W), F(h^{-1}(0)), F(h^{-1}(1)))$, and
			
			\item $F \circ \phi|_{h^{-1}(1)} = F|_{h^{-1}(0)}$.
			
		\end{enumerate} 	
			
		\end{proposition}

		\begin{exmp} \label{brancexmp}
			
			Let $K^{n-2}$ be a closed oriented submanifold of a closed oriented manifold $L^n$ with trivial normal bundle and let $K$ be nullhomologous in $L$. Thus, there exists a Morse open book $(W, h, \phi)$ on $L$, such that the closure of $L \setminus K$ is the relative cobordism $W$ between a regular page $V$ to itself and the monodromy or the return map is $\phi$. We partition the trivial mapping cylinder $(L \times [-1,1]) \times [0,k]$ in $k$ parts : $L \times [-1,1] \times [i,i+1]$ ($0 \leq i \leq k-1$). By Proposition \ref{morsepagetrick}, we can embed a copy of $W$ in $L \times [-1,1] \times [i,i+1]$ via an embedding $F_i$ such that the projection map $L \times [-1,1] \times [i,i+1] \rightarrow [i,i+1]$ restricts to a Morse function on the triad $(F_i(W), F_i(h^{-1}(0)), F_i(h^{-1}(1)))$, and $F_i \circ \phi|_{h^{-1}(1)} = F_i|_{h^{-1}(0)}$. By concatenating $F_0, F_1,\dots, F_{k-1}$, we get a spun embedding of a Morse open book $(W_k, h_k, \phi^k)$ in $\ob(L \times [-1,1],\id) = L \times \s^2$, where $W_k$ is a relative cobordism obtained by concatenating $k$ copies of $W$ and $h_k$ is the corresponding Morse function given by $h \ast h \ast \cdots \ast h$ ($k$ times). By construction, the total space $\widetilde{L}$ of the Morse open book $(W_k, h_k, \phi^k)$ is a $k$-fold cyclic branched cover of $L$ branched along $K$.
			
		\end{exmp}

\end{document}

%% file: fig01.pdf_tex
\begingroup%
  \makeatletter%
  \providecommand\color[2][]{%
    \errmessage{(Inkscape) Color is used for the text in Inkscape, but the package 'color.sty' is not loaded}%
    \renewcommand\color[2][]{}%
  }%
  \providecommand\transparent[1]{%
    \errmessage{(Inkscape) Transparency is used (non-zero) for the text in Inkscape, but the package 'transparent.sty' is not loaded}%
    \renewcommand\transparent[1]{}%
  }%
  \providecommand\rotatebox[2]{#2}%
  \newcommand*\fsize{\dimexpr\f@size pt\relax}%
  \newcommand*\lineheight[1]{\fontsize{\fsize}{#1\fsize}\selectfont}%
  \ifx\svgwidth\undefined%
    \setlength{\unitlength}{602.95882516bp}%
    \ifx\svgscale\undefined%
      \relax%
    \else%
      \setlength{\unitlength}{\unitlength * \real{\svgscale}}%
    \fi%
  \else%
    \setlength{\unitlength}{\svgwidth}%
  \fi%
  \global\let\svgwidth\undefined%
  \global\let\svgscale\undefined%
  \makeatother%
  \begin{picture}(1,0.37408146)%
    \lineheight{1}%
    \setlength\tabcolsep{0pt}%
    \put(0,0){\includegraphics[width=\unitlength,page=1]{fig01.pdf}}%
    \put(0.01125819,0.33384976){\color[rgb]{0,0,0}\makebox(0,0)[lt]{\lineheight{1.25}\smash{\begin{tabular}[t]{l}$M^n \times \{1\}$\end{tabular}}}}%
    \put(0,0){\includegraphics[width=\unitlength,page=2]{fig01.pdf}}%
    \put(-0.00070034,0.02856437){\color[rgb]{0,0,0}\makebox(0,0)[lt]{\lineheight{1.25}\smash{\begin{tabular}[t]{l}$M^n \times \{0\}$\end{tabular}}}}%
    \put(0,0){\includegraphics[width=\unitlength,page=3]{fig01.pdf}}%
    \put(0.91961992,0.33611368){\color[rgb]{0,0,0}\makebox(0,0)[lt]{\lineheight{1.25}\smash{\begin{tabular}[t]{l}$M^n \times \{1\}$\end{tabular}}}}%
    \put(0.91681645,0.02446371){\color[rgb]{0,0,0}\makebox(0,0)[lt]{\lineheight{1.25}\smash{\begin{tabular}[t]{l}$M^n \times \{0\}$\end{tabular}}}}%
    \put(0.90451449,0.20625948){\color[rgb]{0,0,0}\makebox(0,0)[lt]{\lineheight{1.25}\smash{\begin{tabular}[t]{l}$M^n \times \{t\}$\end{tabular}}}}%
    \put(0.33701157,0.02057415){\color[rgb]{0,0,0}\makebox(0,0)[lt]{\lineheight{1.25}\smash{\begin{tabular}[t]{l}$\Sigma$\end{tabular}}}}%
    \put(0.73650425,0.10336539){\color[rgb]{0,0,0}\makebox(0,0)[lt]{\lineheight{1.25}\smash{\begin{tabular}[t]{l}$f(\Sigma)$\end{tabular}}}}%
  \end{picture}%
\endgroup%

%% file: Hopf_band.pdf_tex
\begingroup%
  \makeatletter%
  \providecommand\color[2][]{%
    \errmessage{(Inkscape) Color is used for the text in Inkscape, but the package 'color.sty' is not loaded}%
    \renewcommand\color[2][]{}%
  }%
  \providecommand\transparent[1]{%
    \errmessage{(Inkscape) Transparency is used (non-zero) for the text in Inkscape, but the package 'transparent.sty' is not loaded}%
    \renewcommand\transparent[1]{}%
  }%
  \providecommand\rotatebox[2]{#2}%
  \newcommand*\fsize{\dimexpr\f@size pt\relax}%
  \newcommand*\lineheight[1]{\fontsize{\fsize}{#1\fsize}\selectfont}%
  \ifx\svgwidth\undefined%
    \setlength{\unitlength}{335.27064575bp}%
    \ifx\svgscale\undefined%
      \relax%
    \else%
      \setlength{\unitlength}{\unitlength * \real{\svgscale}}%
    \fi%
  \else%
    \setlength{\unitlength}{\svgwidth}%
  \fi%
  \global\let\svgwidth\undefined%
  \global\let\svgscale\undefined%
  \makeatother%
  \begin{picture}(1,0.3300828)%
    \lineheight{1}%
    \setlength\tabcolsep{0pt}%
    \put(0,0){\includegraphics[width=\unitlength,page=1]{Hopf_band.pdf}}%
    \put(0.20602709,0.28022802){\color[rgb]{0,0,0}\makebox(0,0)[lt]{\lineheight{1.25}\smash{\begin{tabular}[t]{l}$+$\end{tabular}}}}%
    \put(0.7680343,0.24995874){\color[rgb]{0,0,0}\makebox(0,0)[lt]{\lineheight{1.25}\smash{\begin{tabular}[t]{l}$-$\end{tabular}}}}%
    \put(-0.16097568,0.47801636){\color[rgb]{0,0,0}\makebox(0,0)[lt]{\begin{minipage}{1.16189242\unitlength}\raggedright \end{minipage}}}%
    \put(0.03036771,0.36044391){\color[rgb]{0,0,0}\makebox(0,0)[lt]{\begin{minipage}{0.97285442\unitlength}\raggedright \end{minipage}}}%
    \put(0,0){\includegraphics[width=\unitlength,page=2]{Hopf_band.pdf}}%
    \put(0.1517657,0.02355307){\color[rgb]{0,0,0}\makebox(0,0)[lt]{\lineheight{1.25}\smash{\begin{tabular}[t]{l}$c$\end{tabular}}}}%
    \put(0.72719494,0.01511156){\color[rgb]{0,0,0}\makebox(0,0)[lt]{\lineheight{1.25}\smash{\begin{tabular}[t]{l}$c$\end{tabular}}}}%
  \end{picture}%
\endgroup%

%% file: fig02.pdf_tex
\begingroup%
  \makeatletter%
  \providecommand\color[2][]{%
    \errmessage{(Inkscape) Color is used for the text in Inkscape, but the package 'color.sty' is not loaded}%
    \renewcommand\color[2][]{}%
  }%
  \providecommand\transparent[1]{%
    \errmessage{(Inkscape) Transparency is used (non-zero) for the text in Inkscape, but the package 'transparent.sty' is not loaded}%
    \renewcommand\transparent[1]{}%
  }%
  \providecommand\rotatebox[2]{#2}%
  \newcommand*\fsize{\dimexpr\f@size pt\relax}%
  \newcommand*\lineheight[1]{\fontsize{\fsize}{#1\fsize}\selectfont}%
  \ifx\svgwidth\undefined%
    \setlength{\unitlength}{639.32209268bp}%
    \ifx\svgscale\undefined%
      \relax%
    \else%
      \setlength{\unitlength}{\unitlength * \real{\svgscale}}%
    \fi%
  \else%
    \setlength{\unitlength}{\svgwidth}%
  \fi%
  \global\let\svgwidth\undefined%
  \global\let\svgscale\undefined%
  \makeatother%
  \begin{picture}(1,0.5519163)%
    \lineheight{1}%
    \setlength\tabcolsep{0pt}%
    \put(0,0){\includegraphics[width=\unitlength,page=1]{fig02.pdf}}%
    \put(0.00569692,0.09027263){\color[rgb]{0,0,0}\makebox(0,0)[lt]{\lineheight{1.25}\smash{\begin{tabular}[t]{l}$M \times \{-2\}$\end{tabular}}}}%
    \put(-0.000552,0.22403902){\color[rgb]{0,0,0}\makebox(0,0)[lt]{\lineheight{1.25}\smash{\begin{tabular}[t]{l}$M \times \{-1\}$\end{tabular}}}}%
    \put(0.01106477,0.38171443){\color[rgb]{0,0,0}\makebox(0,0)[lt]{\lineheight{1.25}\smash{\begin{tabular}[t]{l}$M \times \{ 0\}$\end{tabular}}}}%
    \put(0.02215883,0.48657019){\color[rgb]{0,0,0}\makebox(0,0)[lt]{\lineheight{1.25}\smash{\begin{tabular}[t]{l}$M \times \{1\}$\end{tabular}}}}%
    \put(0.2481109,0.01251296){\color[rgb]{0,0,0}\makebox(0,0)[lt]{\lineheight{1.25}\smash{\begin{tabular}[t]{l}$s=0$\end{tabular}}}}%
    \put(0.5277465,0.0030611){\color[rgb]{0,0,0}\makebox(0,0)[lt]{\lineheight{1.25}\smash{\begin{tabular}[t]{l}$s=\frac{1}{3}$\end{tabular}}}}%
    \put(0.83804562,0.01049719){\color[rgb]{0,0,0}\makebox(0,0)[lt]{\lineheight{1.25}\smash{\begin{tabular}[t]{l}$s=\frac{2}{3}$\end{tabular}}}}%
  \end{picture}%
\endgroup%